\documentclass[conference]{IEEEtran}
\IEEEoverridecommandlockouts
\usepackage{cite}
\usepackage{amsthm,amsmath,amssymb,amsfonts}
\usepackage{algorithmic}
\usepackage{graphicx}
\usepackage{textcomp}
\usepackage{mfl}
\def\BibTeX{{\rm B\kern-.05em{\sc i\kern-.025em b}\kern-.08em
    T\kern-.1667em\lower.7ex\hbox{E}\kern-.125emX}}

\newtheorem{Def}{\sc Definition}
\newtheorem{Thm}{\sc Theorem}

\newtheorem{Cor}[Thm]{\sc Corollary}

\begin{document}

\title{Saturated Models in Mathematical Fuzzy Logic*\\
\thanks{This is a pre-print of a paper published as: Saturated Models in Mathematical Fuzzy Logic. \emph{ Proceedings of the IEEE International  Symposium on Multiple-Valued Logic 2018},  IEEE Computer Society: 150-155.
We are indebted to the anonymous referees for their helpful comments. Guillermo Badia is supported by the project I 1923-N25 of the Austrian Science Fund (FWF). Carles Noguera is supported by the project GA17-04630S of the Czech Science Foundation (GA\v{C}R) and has also received funding from the European Union's Horizon 2020
research and innovation programme under the Marie Sklodowska-Curie grant
agreement No 689176 (SYSMICS project).}
}

\author{\IEEEauthorblockN{Guillermo Badia}
\IEEEauthorblockA{\textit{Department of Knowledge-Based Mathematical Systems} \\
\textit{Johannes Kepler University Linz}\\
Linz, Austria \\
guillebadia89@gmail.com}
\and
\IEEEauthorblockN{Carles Noguera}
\IEEEauthorblockA{\textit{Institute of Information Theory and Automation} \\
\textit{Czech Academy of Sciences}\\
Prague, Czech Republic\\
noguera@utia.cas.cz}
}

\maketitle

\begin{abstract}
This paper considers the problem of building saturated models for first-order graded logics. We define types as pairs of sets of formulas in one free variable which express properties that an element is expected, respectively, to satisfy and to falsify. We show, by means of an elementary chains construction, that each model can be elementarily extended to a saturated model where as many types as possible are realized. In order to prove this theorem we obtain, as by-products, some results on tableaux (understood as pairs of sets of formulas) and their consistency and satisfiability, and a generalization of the Tarski--Vaught theorem on unions of elementary chains.
\end{abstract}


\section{Introduction}
Mathematical fuzzy logic studies graded logics as particular kinds of many-valued inference systems in several formalisms, including first-order predicate languages. Models of such first-order graded logics are variations of classical structures in which predicates are evaluated over wide classes of algebras of truth degrees, beyond the classical two-valued Boolean algebra. Such models are relevant for recent computer science developments in which they are studied as {\em weighted structures} (see e.g.~\cite{HoMoVi:VCSP}).

The study of models of first-order fuzzy logics is based on the corresponding strong completeness theorems \cite{cin,cin1} and has already addressed several crucial topics such as: characterization of completeness properties with respect to models based on particular classes of algebras~\cite{Cintula-EGGMN:DistinguishedSemantics}, models of logics with evaluated syntax~\cite{Novak-Perfilieva-Mockor:2000,MuNo}, study of mappings and diagrams~\cite{de3}, ultraproduct constructions~\cite{Dellunde:RevisitingUltraproducts, de4}, characterization of elementary equivalence in terms of elementary mappings~\cite{de}, characterization of elementary classes as those closed under elementary equivalence and ultraproducts~\cite{de4}, L\"owenheim--Skolem theorems~\cite{de2}, and back-and-forth systems for elementary equivalence \cite{de5}. A related stream of research is that of continuous model theory~\cite{Chang-Keisler:ContinuousModelTheory,Caicedo:StrengthLukasiewicz}.

Another important item in the classical agenda is that of saturated models, that is, the construction of structures rich in elements satisfying many expressible properties. In continuous model theory the construction of such models is well known (cf. \cite{ben}). However, the problem has not yet been addressed in mathematical fuzzy logic, but only formulated in~\cite{de4}, where Dellunde suggested that saturated models of fuzzy logics could be built as an application of ultraproduct constructions. This idea followed the classical tradition found in~\cite{ck}. However, in other classical standard references such as \cite{h,m, sacks} the construction of saturated structures is obtained by other methods. The goal of the present paper is to show the existence of saturated models for first-order graded logics by means of an elementary construction. 

The paper is organized as follows: after this introduction, Section~\ref{s:prelim} presents the necessary preliminaries we need by recalling several semantical notions from mathematical fuzzy logic, namely, the algebraic counterpart of extensions of the uninorm logic \UL, fuzzy first-order models based on such algebras, and some basic model-theoretic notions. Section~\ref{s:Tableaux} introduces the notion of tableaux (necessary for our treatment of types) as pairs of sets of formulas and proves that each consistent tableau has a model. Section~\ref{s:Main} defines types as pairs of sets of formulas in one free variable (roughly speaking, expressing the properties that an element should satisfy and falsify) and contains the main results of the paper: a fuzzy version of the Tarski-Vaught theorem for unions of elementary chains and the existence theorem for saturated models. Finally, Section~\ref{s:Conclusions} ends the paper with some concluding remarks.

\section{Preliminaries}\label{s:prelim}

In this section we introduce the object of our study, fuzzy first-order models, and several necessary related notions for the development of the paper. For comprehensive information on the subject, one may consult the handbook of Mathematical Fuzzy Logic~\cite{Cintula-FHN:HBMFL} (e.g. Chapters 1 and 2).

We choose, as the underlying propositional basis for the first-order setting, the class of residuated uninorm-based logics~\cite{Metcalfe-Montagna:SubstructuralFuzzy}. This class contains most of the well-studied particular systems of fuzzy logic that can be found in the literature and has been recently proposed as a suitable framework for reasoning with graded predicates in~\cite{Cintula-Noguera-Smith:GradedLORI}, while it  retains important properties, such as associativity and commutativity of the residuated conjunction, that will be used to obtain the results of this paper.

The algebraic semantics of such logics is based on {\em \UL-algebras}, that is, algebraic structures in the language $\mathcal{L} = \{\wedge,\vee,\conj,\to,\0,\1,\bot,\top\}$ of the form
$\alg{A}=\tuple{A,\wedge^\alg{A},\vee^\alg{A},\conj^\alg{A},\to^\alg{A},\0^\alg{A},\1^\alg{A},\bot^\alg{A},\top^\alg{A}}$ such that
\begin{flushleft}
\begin{itemize}
\item $\tuple{A,\wedge^\alg{A},\vee^\alg{A},\bot^\alg{A},\top^\alg{A}}$ is a bounded lattice,
\item $\tuple{A,\conj^\alg{A},\1^\alg{A}}$ is a commutative monoid,
\item for each $a,b,c\in A$, we have:
\begin{align*}
& a\conj^\alg{A} b\leq c  \quad \mathrm{iff }  \quad  b\leq a \to^\alg{A} c, & \mathrm{(res)}\\
&((a \to^\alg{A} b)\land\1^\alg{A}) \lor^\alg{A} ((b \to^\alg{A} a)\land^\alg{A}\1^\alg{A})  = \1^\alg{A} & \mathrm{(lin)}
\end{align*}
\end{itemize}
\end{flushleft}

$\alg{A}$ is called a {\em \UL-chain} if its underlying lattice is linearly ordered. {\em Standard} \UL-chains are those define over the real unit interval $[0,1]$ with its usual order; in that case the operation $\conj^\alg{A}$ is a residuated uninorm, that is, a left-continuous binary associative commutative monotonic operation with a neutral element $\1^\alg{A}$ (which need not coincide with the value $1$).

Let $\fm$ denote the set of propositional formulas written in the language of \UL-algebras with a denumerable set of variables and let $\FM$ be the absolutely free algebra defined on such set. Given a \UL-algebra $\alg{A}$, we say that an {\em $\alg{A}$-evaluation} is a homomorphism from $\FM$ to $\alg{A}$.  The logic of all \UL-algebras is defined by establishing, for each $\Gamma \cup \{\f\} \subseteq \fm$, $\Gamma \models \f$ if and only if, for each \UL-algebra $\alg{A}$ and each $\alg{A}$-evaluation $e$, we have $e(\f) \geq \1^\alg{A}$, whenever $e(\p) \geq \1^\alg{A}$ for each $\p \in \Gamma$. The logic \UL\ is, hence, defined as preservation of truth over all \UL-algebras, where the notion of truth is given by the set of designated elements, or {\em filter}, $\mathcal{F}^{\alg{A}} = \{a \in A \mid a \geq \1^\alg{A}\}$. The standard completeness theorem of \UL\ proves that the logic is also complete with respect to its intended semantics: the class of \UL-chains defined over $[0,1]$ by residuated uninorms (the standard \UL-chains); this justifies the name of \UL\ (uninorm logic).

Most well-known propositional fuzzy logics can be obtained by extending \UL\  with additional axioms and rules (in a possibly expanded language). Important examples are G\"odel--Dummett logic \G\ and \L ukasiewicz logic $\mathrmL$. 

A {\em predicate language} $\pl$ is a triple $\tuple{\mathbf{P,F,ar}}$, where $\mathbf{P}$ is a non-empty set of predicate symbols, $\mathbf{F}$ is a set of function symbols, and $\mathbf{ar}$ is a function assigning to each symbol a natural number called the \emph{arity} of the symbol. Let us further fix a denumerable set $V$ whose elements are called {\em object variables}. The sets of {\em $\pl$-terms}, {\em atomic $\pl$-formulas}, and {\em $\tuple{\lang{L},\pl}$-formulas} are defined as in classical logic. A {\em \pl-structure} $\model{M}$ is a pair $\tuple{\alg{A},\struct{M}}$ where $\alg{A}$ is a \UL-chain and $\struct{M} = \tuple{M,\Tuple{P_\struct{M}}_{P\in\mathbf{P}},\Tuple{F_\struct{M}}_{F\in\mathbf{F}}}$, where
$M$ is a non-empty domain; $P_\struct{M}$ is a function $M^n\to A$, for each $n$-ary predicate symbol $P\in\mathbf{P}$; and $F_\struct{M}$ is a function $M^n\to M$ for each
$n$-ary function symbol $F\in\mathbf{F}$. An {\em $\model{M}$-evaluation} of the object variables is a mapping $\eval{v} \colon V \to M$; by $\eval{v}[x{\to}a]$ we denote the $\model{M}$-evaluation where $\eval{v}[x{\to}a](x) = a$ and $\eval{v}[x{\to}a](y) = \eval{v}(y)$ for each object variable $y\neq x$. We define the {\em values} of the terms and the {\em truth values} of the formulas as (where for ${\circ}$ stands for any $n$-ary connective in $\lang{L}$):
\begin{center}
\begin{tabular}{rcll}
$\semvalue{x}^\model{M}_\eval{v}$ & $=$ & $\eval{v}(x) ,$
\\[0.5ex]
$\semvalue{F(\vectn{t})}^\model{M}_\eval{v}$ & $=$ & $F_\struct{M}(\semvalue{t_1}^\model{M}_\eval{v}\!,\,\dots
,\,\semvalue{t_n}^\model{M}_\eval{v}),$
\\[0.5ex]
$\semvalue{P(\vectn{t})}^\model{M}_\eval{v}$ & $=$ &
$P_\struct{M}(\semvalue{t_1}^\model{M}_\eval{v}\!,\,\dots ,\,\semvalue{t_n}^\model{M}_\eval{v}),$ &
\\[0.5ex]
$\semvalue{{\circ}(\vectn{\f})}^\model{M}_\eval{v}$ & $=$ & ${\circ}^\alg{A}(\semvalue{\f_1}^\model{M}_\eval{v}\!,\,\dots ,\,\semvalue{\f_n}^\model{M}_\eval{v}),$ &
\\[0.5ex]
$ \semvalue{\All{x}\f}^\model{M}_\eval{v}$ & $=$ & $\inf_{\leq_\alg{A}}
\{\semvalue{\f}^{\model{M}}_{\eval{v}[x{\to}m]} \mid m\in M\},$
\\[0.5ex]
$ \semvalue{\Exi{x}\f}^\model{M}_\eval{v}$ & $=$ & $\sup_{\leq_\alg{A}}
\{\semvalue{\f}^{\model{M}}_{\eval{v}[x{\to}m]} \mid m\in M\}.$
\end{tabular}
\end{center}
If the infimum or supremum does not exist, the corresponding value is
undefined. We say that $\model{M}$ is a {\em safe} if $\semvalue{\f}^{\model{M}}_{\eval{v}}$ is defined for each\/ $\pl$-formula $\f$ and each\/ $\model{M}$-evaluation $\eval{v}$. Formulas without free variables are called {\em sentences} and a set of sentences is called a {\em theory}. Observe that if $\f$ is a sentence, then its value does not depend on a particular $\model{M}$-evaluation; we denote its value as $\semvalue{\f}_{\struct{M}}^{\alg{A}}$. If $\f$ has free variables among $\{x_1,\ldots,x_n\}$ we will denote it as $\f(x_1,\ldots,x_n)$; then the value of the formula under a certain evaluation $\eval{v}$ depends only on the values given to the free variables; if $v(x_i) = d_i \in M$ we denote $\semvalue{\f}^\model{M}_\eval{v}$ as $\semvalue{\f(d_1,\ldots,d_n)}^\alg{A}_\struct{M}$. We say that $\model{M}$ is a {\em model} of a theory $T$, in symbols $\model{M} \models T$, if it is safe and for each $\f \in T$, $\semvalue{\f}_{\struct{M}}^{\alg{A}} \geq \1^\alg{A}$.

Observe that we allow arbitrary UL-chains and we do not focus in any kind of standard completeness properties.

Using the semantics just defined, the notion of semantical consequence is lifted from the propositional to the first-order level in the obvious way. Such first-order logics satisfy two important properties that we will use in the paper (see e.g.~\cite{Cintula-Horcik-Noguera:Quest}), for each theory $T \cup \{\f,\p,\x\}$ (inductively defining for each formula $\alpha$: $\alpha^0 = \1$, and for each natural $n$, $\alpha^{n+1}= \alpha^n \conj \alpha$):

\begin{enumerate}
\item Local deduction theorem: $T, \f \models \p$ if, and only if, there is a natural number $n$ such that $T \models (\f \land \1)^n \to \p$.
\item Proof by cases: If $T, \f \models \x$ and $T, \p \models \x$, then\\$T, \f \lor \p \models \x$.
\item Consequence compactness: If $T \vDash \f$, then for some finite $T_0 \subseteq T$, $T_0 \vDash \f$.
\end{enumerate}
Observe that alternatively we could have introduced calculi and a corresponding notion of deduction $\vdash$ for these logics, but we prefer to keep the focus of the paper on the semantics.

\section{Tableaux}\label{s:Tableaux}
A \emph{tableau} is a pair $\tuple{T, U}$ such that $T$ and $U$ are sets of formulas. A tableau $\tuple{T_0, U_0}$ is called a {\em subtableau} of $\tuple{T, U}$ if $T_0 \subseteq T$ and $U_0 \subseteq U$. $\tuple{T, U}$ is \emph{satisfied } by a model  $\model{M}= \tuple{\alg{ A}, {\bf M}}$, if there is an $\model{M}$-evaluation $\eval{v}$ such that for each $\f \in T$, $\semvalue{\f}^{\model{M}}_{\eval{v}}\geq \1^\alg{A}$, and for all $\p \in U$, $\semvalue{\p}^{\model{M}}_{\eval{v}} < \1^\alg{A}$. Also, we write  $\tuple{T, U} \models \f$ meaning that for any model and evaluation that satisfies $\tuple{T, U}$, the model and the evaluation must make $\f$ true as well. A tableau $\tuple{T, U}$ is said to be \emph{consistent} if there is no  finite $U_0 \subseteq U$ such that $T \models \bigvee U_0$. In the extreme case,  we define $\bigvee \emptyset$ as $\bot$. Our choice of terminology here comes from \cite{cha}, where such tableaux are introduced for the intuitionistic setting, where Boolean negation is also absent. The intuitive idea is that in a semantic tableau as we go along we try to make everything on the left true while falsifying everything on the right.

Following~\cite{cin}, we say that a set of sentences $T$ is a {\em $\exists$-Henkin theory} if, whenever $T \models \Exi{x} \f (x)$, there is a constant $c$ such that $T \models \f (c)$. $T$ is a {\em Henkin theory} if $T \not\models \All{x} \f (x)$ implies that there is a constant $c$ such that $T \not\models \f (c)$. $T$ is {\em doubly Henkin} if it is both $\exists$-Henkin and Henkin. $T$ is a {\em linear theory} if for any pair of sentences $\f, \psi$ either $T \models \f \rightarrow \psi$ or $T \models \psi \rightarrow \f$.

The following result shows that each consistent tableau has a model, which will be necessary in the next section.

\begin{Thm}\emph{(Model Existence Theorem)}  Let $\tuple{T, U}$ be a consistent tableau.  Then there is a model that satisfies  $\tuple{T, U}$. \end{Thm}

\begin{proof} We will prove this for countable languages, though the generalization to arbitrary cardinals is straightforward. We start by adding a countable set $C$ of new constants to the language. We enumerate as $\f _0, \f_1, \f_2, \dots$ all the formulas of the expanded language, and as $\tuple{\theta _0, \psi_0}, \tuple{\theta_1, \psi_1}, \tuple{\theta_2, \psi_2}, \dots$ all pairs of such formulas. We modify the proofs of Theorem 4 and Lemma 2 from \cite{cin} by building two chains of theories $T_0 \subseteq \dots \subseteq T_n \subseteq \dots$ and $U_0 \subseteq \dots \subseteq U_n \subseteq \dots$ such that $\tuple{\bigcup_{i < \omega}T_i, \bigcup_{i<\omega} U_i}$ is a consistent tableau (checking that at every stage we obtain a consistent tableau $\tuple{T_i, U_i}$), plus $\bigcup_{i < \omega}T_i$ is a linear doubly Henkin theory. Then, we will simply construct the canonical model as in Lemma 3 from \cite{cin}. We proceed by induction:

{\sc Stage $0: $} Define $T_0 = T$ and $U_0 = U$.

{\sc Stage $s+1 = 3i+1: $} At this stage, we make sure that our final theory will be Henkin. To this end we follow the proof of Lemma 2 (1) from \cite{cin}. If $\f_i$ is not of the form $\All{x} \chi (x)$, then define $T_{s+1}=T_s$ and $U_{s+1}=U_s$. Assume now that  $\f_i =\All{x} \chi (x)$. Then, we consider the following two cases:

\begin{itemize}
\item[(i)] There is a finite $U'_s \subseteq U_s$ such that $T_s \models (\bigvee U'_s)  \vee \All{x} \chi (x)$. Then, we define $T_{s+1}=T_s \cup \{\All{x} \chi (x)\}$ and $U_{s+1}=U_s$.

\item[(ii)] Otherwise, let $T_{s+1}=T_s$  and $U_{s+1} = U_s \cup \{ \chi(c) \}$ (where $c$ is the first unused constant from $C$ up to this stage).
\end{itemize}

We have to check that $\tuple{T_{s+1},  U_{s+1}}$ is consistent in both cases. Suppose that (i) holds and that $T_s \cup \{\All{x} \chi (x)\} \models \bigvee U'_{s} $ for some finite $U'_{s} \subseteq U_{s}$ . By construction, we must have that $T_s \models (\bigvee U''_{s} ) \vee \All{x} \chi (x)$ for some finite $U''_{s} \subseteq U_{s}$. Take the finite set $\overline{U_s} = U'_s \cup U''_s$; clearly we also have $T_s \models (\bigvee \overline{U_{s}} ) \vee \All{x} \chi (x)$. Now, by the local deduction theorem, $T_s \models (\All{x} \chi (x) \wedge \bar{1})^n \rightarrow \bigvee \overline{U_{s}}$ for some $n$. On the other hand, $T_s \models \bigvee \overline{U_{s}}  \rightarrow \bigvee \overline{U_{s}}$. Now, recall that $\All{x} \chi (x) \models  (\All{x} \chi (x) \wedge \bar{1})^n$ (this follows from the rules $\f \models \f \wedge \bar{1}$ and $\f, \psi \models \f \conj \psi $). So, by proof by cases, we have that $T_s \cup \{ (\bigvee \overline{U_{s}}) \vee \All{x} \chi (x)\} \models \bigvee \overline{U_{s}}$, which means that $T_s  \models \bigvee \overline{U_{s}}$, a contradiction since by induction hypothesis $\tuple{T_s, U_s}$ is consistent. If (ii) holds, suppose that $\tuple{T_s, U_s \cup \{ \chi(c) \}}$ is not consistent; then,  $T_s \models (\bigvee U'_{s} )  \vee \chi(c) $  for some finite $U'_{s}  \subseteq U_{s}$. Quantifying away the new constant $c$, we must have $T_s \models \All{x}((\bigvee U'_{s} )  \vee \chi(x)) $, so   $T_s \models  (\bigvee U'_{s} )  \vee  (\All{x} \chi(x)) $, which contradicts the fact that we are considering case (ii).

{\sc Stage $s+1 = 3i+2: $}  At this stage we make sure that we will eventually obtain an $\exists$-Henkin theory.  If $\f_i$ is not of the form $\Exi{x} \chi (x)$, then let $T_{s+1}=T_s$ and $U_{s+1}=U_s$. Otherwise, as in Lemma 2 (2) from \cite{cin}, we have two cases to consider:

\begin{itemize}
\item[(i)] There is a finite $U'_s \subseteq U_s$ such that $T_s \cup\{\f_i\} \models \bigvee U'_s $, then we define $T_{s+1}=T_s $ and $U_{s+1}=U_s$.

\item[(ii)] Otherwise, define $T_{s+1}=T_s \cup \{\chi(c)\}$  (where $c$ is the first unused constant from $C$) and $U_{s+1} = U_s $.
\end{itemize}

Again, in both cases $\tuple{T_{s+1}, U_{s+1}}$ is consistent (check the proof of Lemma 2 (2) from~\cite{cin}).

{\sc Stage $s+1 = 3i+3: $} At this stage we work to ensure that our final theory will be linear. So given the pair $\tuple{\theta_i, \psi_i}$ proceed  as in Lemma 2 (3) from \cite{cin}. That is, we start from the assumption that $\tuple{T_s, U_s}$ is consistent and letting $U_{s+1} = U_s$ we look to add one of $\theta_i \rightarrow \psi_i$ or $\psi_i \rightarrow \theta_i$ to $T_s$ to obtain $T_{s+1}$ while making the resulting tableau $\tuple{T_{s+1}, U_{s+1}}$ consistent. Note that if $T_s \cup \{\theta_i \rightarrow \psi_i\} \models \bigvee U_{s+1}^{\prime}$ and  $T_s \cup \{\psi_i \rightarrow \theta_i\} \models \bigvee U_{s+1}^{\prime \prime}$, then $T_s \cup \{\theta_i \rightarrow \psi_i\} \models (\bigvee U_{s+1}^{\prime}) \vee (\bigvee U_{s+1}^{\prime \prime})$ and  $T_s \cup \{\psi_i \rightarrow \theta_i\} \models (\bigvee U_{s+1}^{\prime}) \vee (\bigvee U_{s+1}^{\prime \prime})$.   Hence, $T_s \cup \{(\psi_i \rightarrow \theta_i) \vee (\theta_i \rightarrow \psi_i)\} \models (\bigvee U_{s+1}^{\prime}) \vee (\bigvee U_{s+1}^{\prime \prime})$ by proof by cases, and since   $\models (\psi_i \rightarrow \theta_i) \vee (\theta_i \rightarrow \psi_i)$, we obtain that $T_s  \models (\bigvee U_{s+1}^{\prime}) \vee (\bigvee U_{s+1}^{\prime \prime})$, a contradiction.
\end{proof}

We can already introduce the general notion of type with respect to a given tableau.

\begin{Def}\label{d:type}
A pair of sets of formulas $\tuple{p, p^{\prime}}$ is a {\em type} of a tableau $\tuple{T, U}$ if the tableau $\tuple{T \cup p, U \cup p^{\prime}}$ is satisfiable.
\end{Def}

Let  $S_n(T, U)$ be the collection of all complete $n$-types (that is, pairs $\tuple{p, q}$ in $n$-many free variables such that for any $\phi$, either $\phi \in p$ or $\phi \in q$) of the tableau $\tuple{T, U}$. This is the space of prime filter-ideal pairs of the $n$-Lindebaum algebra of our logic with the quotient algebra constructed by the relation $\phi \equiv \psi $ iff $\tuple{T, U}\vDash \phi \leftrightarrow \psi$.

Given formulas $\sigma$ and $\theta$, we define  $[\tuple{\sigma, \theta}] = \{\tuple{p, p^{\prime}} \in S_n(T, U) \mid \sigma \in p, \theta \in p^{\prime}\}$. Consider now the collection $B = \{[\tuple{\phi, \psi}] \mid \phi, \psi \ \mbox{are  formulas}\}$. Intuitively, this simply contains all the sets of pairs of theories such that $\phi$ is expected to be true while $\psi$ is expected to fail, for any two  formulas $\phi, \psi$. $B$ is the base for a topology on  $S_n(T, U)$ since given $[\tuple{\phi, \psi}],  [\tuple{\phi^{\prime}, \psi^{\prime}}]\in B$, we have that  $[\tuple{\phi \wedge \phi^{\prime}, \psi \vee \psi^{\prime}}] \subseteq  [\tuple{\phi, \psi}] \cap [\tuple{\phi^{\prime}, \psi^{\prime}}] $ and $[\tuple{\phi, \psi}] \cap [\tuple{\phi \wedge \phi^{\prime}, \psi \vee \psi^{\prime}}] \in B$. Then, there is a topology on  $S_n(T, U)$ such that every open set of $T$ is just the union of a collection of sets from $B$.
A topological space is said to be \emph{strongly S-closed} if every family of open sets with the finite intersection property has a non-empty intersection \cite{don}. Moreover, we will say that a space is  \emph{almost strongly S-closed} if every family of \emph{basic} open sets with the finite intersection property has a non-empty intersection. 

\begin{Cor} \emph{(Tableaux almost strong S-closedness)}  Let $\tuple{T, U}$ be a tableau. If every $\tuple{T_0, U_0}$, with $|T_0|, |U_0|$  finite and $T_0 \subseteq T$ and $U_0 \subseteq U$,  is satisfiable, then $\tuple{T, U}$ is satisfied in some model. \end{Cor}

\begin{proof} It suffices to show that $\tuple{T , U}$ is consistent. Suppose otherwise, that is, there is a finite $U_0 \subseteq U$ such that  $T \models \bigvee U_0$. But then for some finite $T_0 \subseteq T$, $T_0 \models  \bigvee U_0$. Moreover, this implies that $\tuple{T_0, \{ \bigvee U_0\}}$ cannot be satisfiable, but this is a contradiction with the fact that $\tuple{T_0, U_0}$ has a model. \end{proof}
 
\section{Main results}\label{s:Main}

Let us start by recalling the notion of (elementary) substructure (see e.g.~\cite{de2}). $\tuple{\alg{A},\struct{M}}$ is a \emph{substructure} of $\tuple{\alg{B},\struct{N}}$ if the following conditions are satisfied:
\begin{enumerate}
\item $M\subseteq N$.
\item For each $n$-ary function symbol $F\in \mathbf{F}$, and elements $\vectn{d}\in M$, $$F_{\struct{M}}(\vectn{d})=F_{\struct{N}}(\vectn{d}).$$
\item $\alg{A}$ is a subalgebra of $\alg{B}$.
\item For every quantifier-free formula $\f(\vectn{x})$, and $\vectn{d} \in M$, $$\semvalue{\f(\vectn{d})}_\struct{M}^\alg{A}= \semvalue{\f(d_1,\ldots,d_n)}_\struct{N}^\alg{B}.$$
\end{enumerate}

Moreover, $\tuple{\alg{A},\struct{M}}$ is an \emph{elementary substructure} of $\tuple{\alg{B},\struct{N}}$ if condition 4 holds for arbitrary formulas. In this case, we also say that $\tuple{\alg{B},\struct{N}}$ is an \emph{elementary extension} of $\tuple{\alg{A},\struct{M}}$. When instead of subalgebra and subset we have a pair of  injections $\tuple{g, f}$ satisfying the corresponding conditions above, we have an \emph{embedding}.

A sequence $\{\tuple{\alg{ A}_i, {\bf M}_i} \mid i < \gamma\}$ of models is  called a \emph{chain} when for all $i<j<\gamma$ we have that $\tuple{\alg{ A}_i , {\bf M}_i}$ is a substructure of  $\tuple{\alg{ A}_j, {\bf M}_j}$. If, moreover, these substructures are elementary, we speak of an \emph{elementary chain}. The \emph{union} of the chain $\{\tuple{\alg{ A}_i, {\bf M}_i} \mid i < \gamma\}$ is the structure $\tuple{\alg{ A}, {\bf M}}$ where $\alg{ A}$ is the classical union model of the classical chain of algebras $\{\tuple{\alg{ A}_i, {\bf M}_i} \mid i < \gamma\}$ while {\bf M} is defined by taking as its domain $\bigcup _{i<\gamma}{ M}_i$, interpreting the constants of the language as they were interpreted in each ${\bf M}_i$ and similarly with the relational symbols of the language. Let us note that since all the classes of algebras under consideration are classically $\forall_1$-axiomatizable, $\alg{ A}$ will always be an algebra of the appropriate sort. Observe as well that ${\bf M}$ is well defined given that $\{\tuple{\alg{ A}_i, {\bf M}_i} \mid i < \gamma\}$ is a chain.

\begin{Thm}\label{t:Unions}\emph{(Tarski-Vaught theorem on unions of elementary chains)} Let $\alg{ A}=\tuple{\alg{ A}, {\bf M}}$ be the union of an elementary chain  $\{\tuple{\alg{ A}_i, {\bf M}_i} \mid i < \gamma\}$. Then, for each sequence $\overline{a}$ of elements of\/ ${\bf M}_i$ and each formula $\f(\overline{x})$, $ \semvalue{\f (\overline{a})}^\alg{ A}_{\struct{M}} =  \semvalue{\f (\overline{a})}^\model{\alg{A}_i}_{\struct{M}_i}$. Moreover, the union  $\alg{ A}=\tuple{\alg{ A}, {\bf M}}$ is a safe structure. \end{Thm}

\begin{proof} We proceed by induction on the complexity of $\f$. When $\f$ is atomic, the result follows by definition of $\alg{ A}$. For any $n$-ary connective ${\circ}$, 
\begin{center}
$\semvalue{{\circ} (\psi_0 (\overline{a}), \dots, \psi_n(\overline{a}))}^\alg{ A}_{\struct{M}}
= {\circ}^\alg{A}(\semvalue{\psi_0(\overline{a}) }^\alg{ A}_{\struct{M}}, \dots, \semvalue{\psi_n(\overline{a}) }^\alg{ A}_{\struct{M}})= {\circ}^\alg{A_i}(\semvalue{\psi_0(\overline{a}) }^\alg{ A_i}_{\struct{M}_i}, \dots, \semvalue{\psi_n(\overline{a}) }^\alg{ A_i}_{\struct{M}_i}) = \semvalue{{\circ} (\psi_0 (\overline{a}), \dots, \psi_n(\overline{a}))}^\alg{ A_i}_{\struct{M}_i}$, 
\end{center}
where the second equality follows by the induction hypothesis and the definition of \alg{ A}.

Let $\f = \Exi{x} \psi$ (the case of $\f = \All{x} \psi$ is analogous). Consider $\semvalue{ \psi (\overline{a }, x)}^\alg{ A}_{\struct{M}}$ for $\overline{b} \in { M}^n$. Take $j > i$ sufficiently large such that $\bar{b} \in {M}_j^n$. By  induction hypothesis, $\semvalue{ \psi (\overline{a}, b)}^\alg{ A_j}_{\struct{M_j}}= \semvalue{ \psi (\overline{a}, b)}^\alg{ A}_{\struct{M}}$. By the elementarity of the chain, $ \semvalue{\Exi{x} \psi (\overline{a})}^\alg{ A_i}_{\struct{M_i}} = \semvalue{\Exi{x} \psi (\overline{a})}^\alg{ A_j}_{\struct{M_j}}$. Hence, $\semvalue{ \psi (\overline{a}, b)}^\alg{ A}_{\struct{M}} \leq^\alg{ A} \semvalue{\Exi{x} \psi (\overline{a})}^\alg{ A_i}_{\struct{M_i}}$. Then $\semvalue{\Exi{x} \psi (\overline{a})}^\alg{ A_i}_{\struct{M_i}}$ is an upper bound for $$ \{\semvalue{ \psi (\overline{a }, x)}^\alg{ A}_{\struct{M}} \mid \overline{b} \in { M}^n\}$$ in \alg{A}. Moreover, suppose that $u$ is another such upper bound in \alg{A}. This means that we can find $j \geqslant i$ such that $u \in A_j$. Then   $u$ is an upper bound in $\alg{A}_j$ of $$ \{\semvalue{ \psi (\overline{a }, x)}^\alg{ A_j}_{\struct{M_j}} \mid \overline{b} \in { M_j}^n\},$$which means that $$\semvalue{\Exi{x} \psi (\overline{a})}^\alg{ A_i}_{\struct{M_i}} = \semvalue{\Exi{x} \psi (\overline{a})}^\alg{ A_j}_{\struct{M_j}}  \leq^{\alg{A}_j}u,$$ so $$ \semvalue{\Exi{x} \psi (\overline{a})}^\alg{ A_i}_{\struct{M_i}}  \leq^{\alg{A}}u.$$ So $$ \semvalue{\Exi{x} \psi (\overline{a})}^\alg{ A_i}_{\struct{M_i}}  =\semvalue{\Exi{x} \psi (\overline{a})}^\alg{ A}_{\struct{M}}.$$

This establishes as well that the union is this chain of models is a safe structure, and hence, a model.
 \end{proof}

A structure $\tuple{\alg{ A}, {\bf M}}$ is said to be \emph{exhaustive} if every element of \alg{ A} is the value of some formula for some tuple of objects from $M$. In the rest of the paper, we will assume that all models are exhaustive. For that purpose we need to make sure that our constructions always give us back exhaustive models. Clearly, the model obtained in the model existence theorem is exhaustive.

\begin{Cor}The union of an elementary chain of exhaustive models is itself exhaustive. \end{Cor}

\begin{proof}
Suppose that $x \in \alg{A}$, then $x \in \alg{A}_i$ for some $i$, so $x = \semvalue{ \f (\overline{a })}^\alg{ A_i}_{\struct{M_i}}$ for some sequence $\bar{a}$ of elements of $M_i$ and formula $\f$, but then $x =\semvalue{ \f (\overline{a })}^\alg{ A}_{\struct{M}}$ by Theorem~\ref{t:Unions}.
\end{proof}

Given a model $\model{M} = \tuple{\alg{ A}, {\bf M}}$ and a collection $D \subseteq M$, we denote by $\mbox{Th}_D (\model{M})$ the theory of $\model{M}$ relative to $D$, that is, the collection of all sentences $\f$ in a language obtained by augmenting with a list of constants to denote the  elements from $D$ such that $\semvalue{\f}_{\struct{M}}^{\alg{A}} \geq \1^\alg{A}$. On the other hand, $\overline{\mbox{Th}}_D (\model{M})$ will simply denote the set-theoretic complement of $\mbox{Th}_D (\model{M})$.

We are finally ready to define the intended notion of type with respect to a model $\model{M}$ (observe that it is a particular case of Definition~\ref{d:type} when the tableau is $\tuple{\mbox{Th}_D (\model{M}),\overline{\mbox{Th}}_D (\model{M})}$).

\begin{Def} Let $\model{M}=\tuple{\alg{ A}, {\bf M}}$ be a model. If $\tuple{p, p^{\prime}}$ is a pair of sets of  formulas in some variable $x$ and parameters over some $D \subseteq M$, we will call $\tuple{p, p^{\prime}}$ a {\em type} of  $\tuple{\alg{ A}, {\bf M}}$ in $D$ if the tableau $\tuple{\mbox{\emph{Th}}_D (\model{M}) \cup p, \overline{\mbox{\emph{Th}}}_D (\model{M}) \cup p^{\prime}}$ is satisfiable (consistent). We will denote the set of  all such types by $S^{\tuple{\alg{ A}, {\bf M}}}(D)$.\end{Def}

The following defnition captures the notion of a model realizing as many types as possible (under a certain cardinal restriction).

\begin{Def} For any cardinal $\kappa$, a model $\model{M}$ is said to be {\em $\kappa$-saturated} if for any $D \subseteq M$ such that $|D| < \kappa$, any type in $S^\model{M}(D)$ is satisfiable in $\model{M}$.\end{Def}

Before we begin the proof of the main result below, we need to recall the notion of the elementary diagram of a structure. Given a model $\tuple{\alg{ A}, {\bf M}}$, by the {\em elementary diagram} of  $\tuple{\alg{ A}, {\bf M}}$, in symbols $\mbox{Eldiag}(\alg{ A}, {\bf M})$, we will denote the theory of $\tuple{\alg{ A}, {\bf M}}$ relative to the whole of $M$. In a nutshell, $\mbox{Eldiag}(\alg{ A}, {\bf M}) = \mbox{Th}_M(\alg{ A}, {\bf M})$. This notion has been studied in detail in \cite{de,de3,cin} and we refer the reader to those papers for further information. On the other hand, $\overline{\mbox{Eldiag}}(\alg{A}, {\bf M})$ will denote the set-theoretic complement of $\mbox{Eldiag}(\alg{A}, {\bf M})$. The important fact for our purposes is that, there is a canonical model (those models given by the model existence theorem) constructed from $\tuple{\mbox{Eldiag}(\alg{ A}, {\bf M}), \overline{\mbox{Eldiag}}(\alg{ A}, {\bf M}) }$ such that we can build an embedding from $\tuple{\alg{ A}, {\bf M}}$ into the new canonical model.

We can observe that in the above definition it suffices to consider types in one free variable. Indeed, the more general case of finitely many variables, say, $x_0, \dots, x_n$ can be reduced to the one variable case by a standard argument. Suppose that $\tuple{\mbox{Th}_D (\model{M}) \cup p, \overline{\mbox{Th}}_D (\model{M}) \cup p^{\prime}}$ is satisfiable in some model $ \tuple{{\bf B} , {\bf N}}$ (obtained by the model existence theorem) by a sequence $e_0, \dots, e_n \in N$. Thus, the type of $e_0$ with parameters over $D$ is realized in $\model{M} =\tuple{\alg{ A}, {\bf M}}$ by an element $e_0^{\prime}$. But then we can also realize in $\tuple{\alg{ A}, {\bf M}}$ the type $\tuple{T, U}$ where
\[
T= \{\f(x,e_0^{\prime} ) \mid \tuple{{\bf B} , {\bf N}} \models \f(e_1,e_0) \} 
\]
\[
U = \{\psi(x,e_0^{\prime} ) \mid \tuple{{\bf B} , {\bf N}}\not \models \psi(e_1,e_0) \}
\]
since it is satisfied in $ \tuple{{\bf B} , {\bf N}}$ by interpreting $e_0^{\prime}$ as $e_0$.
 Keep going this way until we finally realize the type of an element $e_n^{\prime}$ with parameters in $D \cup \{e_0^{\prime}, \dots, e_{n-1}^{\prime}\}$.

Given a collection of theories $\Psi$  of our language and a theory $T$, following Convention 3.22 from \cite{cin1}, we will write $T \Vdash \Psi$ if there is $S \in \Psi$ such that $T \models \f$ for each $\f \in S$.

\begin{Thm}  For each cardinal $\kappa$, each model can be elementarily extended to a $\kappa^+$-saturated model.
\end{Thm}

\begin{proof} Let $\model{M} = \tuple{\alg{A},\struct{M}}$ be a model. Observe that, indeed, 
\begin{center}
$|\{D \subseteq M \mid |D| \leq \kappa\}| \leq |M|^{\kappa}$.
\end{center}
This means, together with that fact that $|S^\model{M}(D)| \leq 2^{\kappa}$, that we can list all   types in $S^\model{M}(D)$ for $D \subseteq M, |D| \leq \kappa$ as $\{\tuple{p_{\alpha}, p_{\alpha^{\prime}}} \mid \alpha < |M|^{\kappa}\}$. 

We can find a model $\tuple{\alg{ A}^{\prime} , {\bf M}^{\prime}}$ that realizes all  types in  $S^\model{M}(D)$ for any $D \subseteq M, |D| \leq \kappa$. We will use the union of elementary chains construction, defining a sequence of models $\{\tuple{\alg{ A}_{\alpha}, {\bf M}_{\alpha}} \mid {\alpha < |M|^{\kappa}}\}$  which is an elementary chain, and where $  \tuple{\alg{ A}_{\alpha}, {\bf M}_{\alpha}}   $ realizes $\tuple{p_{\alpha}, p_{\alpha^{\prime}}}$.

The goal is to 
 build the model $\bigcup_{\alpha < |M|^{\kappa}} \tuple{\alg{ A}_{\alpha}, {\bf M}_{\alpha}}$, which will be our $\tuple{\alg{ A}^{\prime} , {\bf M}^{\prime}}$.

We let

\begin{itemize}
\item[(i)] $\model{M}_{0}=\tuple{\alg{ A}_0, {\bf M}_0}= \tuple{\alg{ A}, {\bf M}}$

\item[(ii)]$\model{M}_{\alpha}=\tuple{\alg{ A}_{\alpha}, {\bf M}_{\alpha}} = \bigcup_{\beta < \alpha} \tuple{\alg{ A}_{\beta} , {\bf M}_{\beta}}$ when $\alpha$ is a limit ordinal.

\item[(iii)] $\model{M}_{\alpha +1}=\tuple{\alg{ A}_{\alpha +1} , {\bf M}_{\alpha+1}}$ is a elementary extension of $\tuple{\alg{ A}_{\alpha}, {\bf M}_{\alpha}}$ which realizes $\tuple{p_{\alpha}, p_{\alpha^{\prime}}}$. We build $\tuple{\alg{ A}_{\alpha +1} , {\bf M}_{\alpha+1}}$  using Lemma 3.24 \cite{cin1}, the construction of canonical models from that paper and our tableaux almost strong S-closedness. In what follows we will use the notation and definitions from \cite{cin1}. We start by showing that   
$$
\mbox{Eldiag}(\alg{ A}_{\alpha}, {\bf M}_{\alpha}) \cup p_{\alpha} \nVdash \{X \},
$$
where $X$ is an arbitrary finite subset of $\overline{\mbox{Eldiag}}(\alg{ A}_{\alpha}, {\bf M}_{\alpha}) \cup  p_{\alpha^{\prime}}$. 
 Observe that the set of theories $\{X\}$ is trivially deductively  directed in the sense of Definition 3.21 from  \cite{cin1}. Using the canonical model construction and Lemma 3.24 \cite{cin1} we can then provide a model for the tableau $\tuple{\mbox{Eldiag}(\alg{ A}_{\alpha}, {\bf M}_{\alpha}) \cup p_{\alpha}, X }$ for each such $X$. Hence, an application of tableaux almost strong S-closedness provides us with a model of $\tuple{\mbox{Eldiag}(\alg{ A}_{\alpha}, {\bf M}_{\alpha}) \cup p_{\alpha}, \overline{\mbox{Eldiag}}(\alg{ A}_{\alpha}, {\bf M}_{\alpha}) \cup  p_{\alpha^{\prime}} }$.

Suppose, for a contradiction, that for each $\x \in X$,

$$
\mbox{Eldiag}(\alg{ A}_{\alpha}, {\bf M}_{\alpha}) \cup p_{\alpha} \models \x.
$$

Then take $\psi \in X$. There are two possibilities. First suppose that $\psi \in \overline{\mbox{Eldiag}}\tuple{\alg{ A}_{\alpha}, {\bf M}_{\alpha}}$. Since  $\mbox{Eldiag}(\alg{ A}_{\alpha}, {\bf M}_{\alpha}) \cup p_{\alpha} \models \psi$, by the local deduction theorem, $ p_{\alpha} \models ( \f \wedge \bar{1} )^n \rightarrow   \psi$ where $\f$ is some lattice conjunction of elements of $\mbox{Eldiag}(\alg{ A}_{\alpha}, {\bf M}_{\alpha})$.
 Quantifying away the new constants (so only constants from the particular $A \subseteq M$  remain), we obtain that $ p_{\alpha} \models \All{\overline{x}} ((\f \wedge \bar{1})^n \rightarrow   \psi)$.
 But then, taking the model of $\tuple{\mbox{Th}_A (\model{M}) \cup p_{\alpha}, \overline{\mbox{Th}}_A (\model{M}) \cup  p_{\alpha^{\prime}} }$,   we get a contradiction because $ \All{\overline{x}} ((\f \wedge \bar{1})^n \rightarrow   \psi)$ would have to be in $\mbox{Th}_A (\model{M}) $, which in turn is contained in $\mbox{Th}_A (\model{M}_{\alpha}) $. But then $\semvalue{\f}_{\model{M}_{\alpha}}^{\alg{ A}_{\alpha}} \geq \1^{\alg{ A}_{\alpha}}$, so $\semvalue{\f \land \bar{1}}_{\model{M}_{\alpha}}^{\alg{ A}_{\alpha}} \geq \1^{\alg{ A_{\alpha}}}$ and hence, $\semvalue{(\f \land \bar{1})^n}_{\model{M}_{\alpha}}^{\alg{ A}_{\alpha}} \geq \1^{\alg{ A}_{\alpha}}$ which leads to a contradiction. On the other hand, suppose that $\psi \in p_{\alpha^{\prime}}$. Similarly, we can obtain that  $\mbox{Eldiag}(\alg{A}_{\alpha}, {\bf M}_{\alpha})  \models \All{\overline{x}} ((\f \land \1)^n \rightarrow \psi)$ where this time $\f$ is a lattice conjunction of elements from $p_{\alpha}$. Then $ \All{\overline{x}} ((\f)^n \rightarrow   \psi)$ would have to be in $\mbox{Th}_A (\model{M}_{\alpha})$, which in turn is contained in $\mbox{Th}_A (\model{M}) $ which would be a contradiction given the model of $\tuple{\mbox{Th}_A (\model{M}) \cup p_{\alpha}, \overline{\mbox{Th}}_A (\model{M}) \cup p_{\alpha^{\prime}} }$. 
\end{itemize}

Next we build another elementary chain to get the $\kappa^+$-saturated structure $\tuple{\alg{D} , \struct{O}}$. This time we put:
\begin{itemize}
\item[(i)] $\tuple{\alg{D}_0 , \struct{O}_0} = \tuple{\alg{A}, {\bf M}}$

\item[(ii)]$\tuple{\alg{D}_{\alpha} ,\struct{O}_{\alpha}} = \bigcup_{\beta < \alpha} \tuple{\alg{D}_{\beta},\struct{O}_{\beta}}$ when $\alpha$ is a limit ordinal.

\item[(iii)] $\tuple{\alg{D}_{\alpha +1} , \struct{O}_{\alpha+1}}$ is a model that elementarily extends $\tuple{\alg{D}_{\alpha} , \alg{O}_{\alpha}}$ and  realizes all types in  $S^{(\alg{D}_{\alpha} , \struct{O}_{\alpha})}(A)$ for any $A \subseteq M_{\alpha}, |A| \leq \kappa$.
\end{itemize}
Consider now $\bigcup_{\alpha < \kappa^+} \tuple{\alg{D}_{\alpha} , \struct{O}_{\alpha}}$, which will be our $\tuple{\alg{D} , \struct{O}}$. Now suppose that $A \subseteq N, |A| \leq \kappa$ and $\tuple{p, p^{\prime}} \in S^{\tuple{\alg{D},\struct{O}}}(A)$. By the regularity of the cardinal $\kappa^+$, we must have that indeed $A \subseteq O_{\alpha}$ for some $\alpha < \kappa^+$. But, of course, since $\mbox{Th}_A(\alg{D},\struct{O}) = \mbox{Th}_A(\alg{D}_{\alpha},\struct{O}_{\alpha})$ and  $\overline{\mbox{Th}}_A(\alg{D},\struct{O}) = \overline{\mbox{Th}}_A(\alg{D}_{\alpha},\struct{O}_{\alpha})$, $\tuple{p, p^{\prime}} \in S^{\tuple{\alg{D}_{\alpha},\struct{O}_{\alpha}}}(A)$, so it is in fact realized in $\tuple{\alg{D}_{\alpha +1},\struct{O}_{\alpha+1}}$, and hence in  $\tuple{\alg{D},\struct{O}}$. 
\end{proof}

\section{Conclusions}\label{s:Conclusions}
In this paper we have shown the existence of saturated models, that is, models realizing as many types as possible (given some cardinality restrictions). A complementary task would be that of building models realizing very few types, which in classical model theory is accomplished by means of the Omitting Types Theorem. Some work has already been started along these lines in the context of mathematical fuzzy logic in~\cite{MuNo,Caicedo:Omitting,cindes}, that have focused on types with respect to a theory. In a forthcoming investigation we plan to extend these works by considering, in the fashion of the present paper, omission of types given by tableaux.

\end{document}